\documentclass[twoside,10pt,fleqn]{article}
\usepackage{amsmath,indentfirst,amsmath,amsfonts,amssymb,amsthm}
\usepackage{CJK,indentfirst,amsmath,amsfonts,amssymb,amsthm,cite}
\usepackage{fancyhdr}
\usepackage[dvips]{graphics}
\usepackage{graphicx}
\usepackage[dvips]{color}
\usepackage{lineno}

\include{extarrows}
\usepackage{extarrows}
\newcommand{\ra}{\rightarrow}
\newcommand{\fr}{\frac}
\newcommand{\lt}{\leqslant}
\newcommand{\gt}{\geqslant}
\newcommand{\bl}{\biggl}
\newcommand{\br}{\biggr}
\def\d{\text{d}}

\def\R{\ensuremath {\mathbb R}}

\renewcommand{\P}{\ensuremath{\mathbb P}}

\newcommand{\E}{\ensuremath{\mathbb E}}
\newcommand{\F}{\ensuremath{\mathcal F}}

\newcommand{\M}{\ensuremath{\mathcal M}}

\newcommand{\A}{\ensuremath{\mathcal A}}
\renewcommand{\L}{\ensuremath{\mathcal{L}}}

\newtheorem{rem}{Remark}[section]

\newtheorem{Def}{Definition}[section]
\newtheorem{theo}{Theorem}[section]

\newtheorem{lem}{Lemma}[section]
\newtheorem{iteration lemma}{iteration Lemma}[section]

\begin{document}

\baselineskip=6mm
%--------------------------The first page's NO. setting--------------------------------------------------
\setcounter{page}{1}
\newcounter{jie}

%%%%%%%%%%%%%%%%%%%%%%%%%%%%%%%%%%%%%%%%%%%%%%%%%%%%%%%%%%%%%%%%%%%%%%%%%%%%%%%%%%%%
%%---------------------Article title and authors-----------------------------------------------------------
\vspace*{0.4cm}
\begin{center}
\baselineskip=6mm
{\LARGE\bf {Dynamics of a mean-reverting stochastic volatility equation with regime switching}}\\[1cm]
\baselineskip=6mm
 \large\bf
 Yanling Zhu$^{a}$, Kai Wang$^{a,b,*}$ and Yong Ren$^{c,*}$
\\[6mm]
 {\footnotesize \it $^{a}$ Department of Applied Mathematics, Anhui University of Finance and Economics, Bengbu 233030,  CHINA\\
 $^{b}$ Center for Applied Mathematics, Tianjin University, Tianjin 300072,  CHINA \\
$^{c}$ Department of Mathematics, Anhui Normal University, Wuhu 241000,   CHINA}
%-----------------footnotes-------------------------------------------------------------------
\begin{figure}[b]
\footnotesize
\rule[-2.truemm]{5cm}{0.1truemm}\\[1mm]%\baselineskip=6.mm
{$^\dagger$Supported  by the NSFC ( 71803001, 11871076), the NSF
of Anhui Province (1708085MA17) and the NSF of Education Bureau of Anhui Province
(KJ2018A0437).
\\ $^*$Corresponding authors.
\\
{\it E-mail addresses}\,:
zhuyanling99\textit{\char64}126.com(YL. Zhu), wangkai050318\textit{\char64}163.com(K. Wang)}, brightry@hotmail.com (Y. Ren)\\
%Revision submitted to Elsevier on \today

\end{figure}
\end{center}
\vspace{0.03cm}
\noindent\hrulefill \newline
\begin{center}
\begin{minipage}{15cm}
\baselineskip=6mm
%-------------------Abstract------------------------------------------------------------------
\vspace{-3mm} \noindent{\bf Abstract:}\ \ In this paper, we consider a mean-reverting stochastic volatility equation with regime switching, and present some sufficient conditions for the existence of global positive solution, asymptotic boundedness in  $p$th moment, positive recurrence and existence of stationary distribution of this equation. Some results obtained in this paper extend the ones in literature. Example is given to verify the results by simulation.
%-------------------Keywords------------------------------------------------------------------

\vspace{1mm}
{\bf AMS(2000):}\ \ 60H10;\ \  60J60;\ \  92D25
\end{minipage}
\end{center}
\vspace{1mm} \noindent\hrulefill \newline
%---------------------------Text-------------------------------------------------------------
%\newpage
\section{Introduction}
Stochastic volatility means that volatility is not a constant, but a stochastic process. By assuming that the volatility of the underlying price is a stochastic process rather
than a constant, it becomes possible to more accurately model derivatives.
Mean-reverting stochastic volatility model is used in the fields
of quantitative finance and financial engineering to evaluate derivative securities,
such as options and swaps. The general mean-reverting stochastic volatility model can be expressed by the following equation:
\begin{equation}\label{GM}
d X_t=(a-bX_t)dt+\sigma X_t^{\theta}dB_t,
\end{equation}
where $X_t$ presents the price or the variance of the price returns of a stock, $a, b$ are positive constants, $\sigma$ is the volatility rate, $\theta$ is a nonnegative constant, $B_t$ is a Brownian motion defined on a complete filtration probability space $(\Omega, \F, \{\F_t\}_{t\geqslant0}, \P)$, and the filtration $\{\F_t\}_{t\geqslant0}$ satisfies the usual condition.
Equation \eqref{GM} has important application in economy, such as the variance of the price returns of stock and the option price(cf.\cite{all-00}). For examples, when  $\theta=0$, it reduces to the mean-reverting Ornstein-Uhlenbeck model,  Mao (1997) \cite{mao-b} discussed the limit distribution of its solution as time $t$ goes to infinity;
 when $\theta=1/2$, it is the square root stochastic variance model,
 which was used by Cox, Ingersoll and  Ross (1985) \cite{cir-85} to express the dynamics of interest rate, and by Heston (1993) \cite{slh-93} to investigate the pricing of a European call option on an asset with stochastic volatility, and as the exchange rate processes by Bates (1996) \cite{bds-96} to investigate  Deutsche mark options; while $\theta=1$, it transforms to the Garch Diffusion model,
 and as $\theta\gt1/2$ Model \eqref{GM} reduces to the Constant Elasticity of Variance model, which provides a relatively simple illustration of many of the effects of volatility explosions; these two stochastic models were presented in Ghysels, Harvey and Renault (1996) \cite{ghr-96} for studying the stochastic volatility in financial markets. For the general case of $\theta\in[1/2,+\infty)$, Mao (1997) \cite{mao-b} shows its solution $X_t>0$ for all $t>0$ almost surely.
\par As we know that in the real world, the socio-economic environment is constantly changing, which means that the parameters $a, b, \sigma$ and $ \theta$ in equation \eqref{GM} will change  as the social and economic environment changes. This type of noise caused by  the change of economic environment is often called  telegraph noise, which can be demonstrated as a switching between two or more regimes of states.  The continuous time Markov chain models the regime switching well. Kazangey and Sworder (1971) \cite{ks-71} presented
a switching system, where a macroeconomic model of the national economy
was used to study the effect of federal housing removal policies on the
stabilization of the housing sector. The term describing the influence of
interest rates was modelled by a finite-state Markov chain to provide a
quantitative measure of the effect of interest rate uncertainty on optimal
policy. Readers can see Mao and Yuan (2006) \cite{mao-06}, and Yin and Zhu (2009) \cite{yz-09} for more details on the theory of switching systems, which are two excellent references on this subject.

\par
 Motivated by these, in this paper we consider a general stochastic volatility equation with regime switching in the following form:
\begin{equation}\label{w1}
\left\{
  \begin{array}{ll}
   d X_t=[a(r_t)-b(r_t)X_t]dt+\sigma(r_t)X_t^{\theta(r_t)}dB_t, \\
(X_0,r_0)=(x_0,\,i_0),
  \end{array}
\right.
\end{equation}
where $a(\cdot),$ $b(\cdot)$, $\sigma(\cdot)$ are positive constants, $\theta(\cdot)\in[1/2,\infty)$, and $(r_t)_{t\geqslant0}$  is a continuous Markov chain taking values in a finite-state space  $\M=\{1,\cdots,m\}$, and with  infinitesimal generator $Q = (q_{ij})\in \mathbb{R}^{m\times m}$. That is  $(r_t)_{t\geqslant0}$ satisfies
\[P(r_{t+ \delta  } = j|r_t = i) = \left\{
                                     \begin{array}{ll}
                                     \quad\;\;\, q_{ij}\delta  + o(\delta ), \;\;\text{if}\; i\not= j,\\
                                      1 + q_{ij}\delta  + o(\delta ),  \;\;\text{if}\;  i=j,
                                     \end{array}
                                   \right. \;\text{as}\;\delta\ra 0^+,
\]
 where $q_{ij}\geqslant0$ is the transition rate from $i$ to $j$ for $i\not=j$,  and $q_{ii}=-\sum_{i\not=j}q_{ij}$, for each $i\in\M.$ We assume that the Brownian motion  $B_t$ is independent of  the Markov chain $(r_t)_{t\geqslant0}$.

 Our contributions in this paper are as follows.
 \vspace{2mm}
\par $\bullet\;$  We show the existence of global almost surely positive solution to equation \eqref{w1} for any initial value $(x_0, i_0)\in \R^+\times\M$, where $\R^+=(0,\infty)$, which extends the corresponding result in Mao et al. \cite{mao0}.
\par $\bullet\;$  We obtain the estimation of the $p$th moment, asymptotic boundedness in $p$th moment and the Lyapunov exponent of the solution $X_t(x_0, i_0)$ to equation \eqref{w1}.
\par $\bullet\;$  We present some sufficient conditions for that the process $(X_t, r_t)$ determined by equation \eqref{w1} is positive recurrence and admits a unique ergodic asymptotically invariant distribution.
 \vspace{2mm}
 \par
 For $V: \R\times\M\longmapsto\R^+$ such that $V(\cdot,i)$ is twice continuously differential with respect to the first variable for each $i\in\M,$ we define the operator $\L$ by
\[
\L V(x,i)=\big[a(i)-b(i)x\big]\fr{\partial V}{\partial x}(x,i)+\fr12 \sigma^2(i)x^{2\theta(i)}\fr{\partial^2 V}{\partial x^2}(x,i)+\sum_{k\in \M}q_{ik}V(x,k).
\]

\section{Global positive solution}
\begin{theo} \label{Th1}
 If for all $i\in\M$, $\theta(i)=1/2$ and $2a(i)\geqslant\sigma^2(i)$; or $\theta(i)>1/2$, then for any $(x_0,i_0)\in \R^+\times\M$, there is a unique solution $X_t(x_0, i_0)$ to equation \eqref{w1} on $t\geqslant0$ and this solution also satisfies $\P(X_t(x_0,i_0)>0, \forall\,t\gt 0)=1.$
\end{theo}
\begin{rem}  If $\#\M=1$($\M$ has only one state), that is, there is no regime switching in  equation \eqref{w1}, Mao et al. \cite{mao-b} investigated the global existence of positive solution to it for the case of  $\theta\geqslant1/2$. For the regime switching equation \eqref{w1} with $\theta(r_t)\equiv\theta$, Mao et al. \cite{mao0} proved the existence of global positive solution to it for the case of $\theta\in[1/2,1]$. For the case of $\#\M>1$, that is, the regime switching  equation \eqref{w1}, Theorem \ref{Th1} presents the same results for $\theta(r_t)\in[1/2,1]$. Moreover, one can see that Theorem \ref{Th1} extends the existence results to the case of $\theta(r_t)>1$. For the case of $\theta(r_t)\in(0,1/2)$, pathwise uniqueness  does not hold, see Girsanov  \cite{Gir}; and for $\theta(r_t)\equiv0,$ the solution of equation \eqref{w1} can be expressed explicitly, which shows that it can be negative.
\end{rem}
\begin{proof} For the case of $\theta\in[1/2,1)$, the coefficients of equation \eqref{w1} obey H$\ddot{o}$lder continuity and  linear growth condition on $(0,+\infty)$; and for the case of $\theta\geqslant1$, the coefficients of equation \eqref{w1}  are locally Lipschitz continuous on $(0,+\infty)$. Thus  for any given initial value  $(x_0,i_0)\in \R^+\times\M$ there is a unique maximal local solution $X_t(x_0,i_0)$ on $[0,\tau_e)$, where $\tau_e$ is the explosion time, see Ikeda \cite{Ikeda-81} for more details on the uniqueness of the solution to stochastic differential equation with  H$\ddot{o}$lder continuous coefficients.
\par Now we show that the solution is globally existent and positive on $\R^+$ a.s..  Let $k_0>0$ be sufficiently large such that $x_0\in[1/k_0, k_0]$, and for each integer $k\geqslant k_0$ define the stopping time
$\tau_k=\inf\{t\in[0,\tau_e)|\; X_t(x_0,i_0)\not\in (1/k,k)\},$ where $\tau_\infty=\lim _{k\ra \infty}\tau_k$. It is obvious that $\tau_\infty\leqslant\tau_e$. So we only need to show that $\tau_\infty=\infty$ a.s., which yields the positiveness of the solution almost surely and the existence of global solution of equation \eqref{w1}. If it is false, then there is a pair of constants $T>0$ and $\epsilon\in(0,1)$ such that
$P\{\tau_\infty\leqslant T\}>\epsilon,$
which yields that there exists an integer $k_1\geqslant k_0$ such that
\begin{equation}\label{n1}
P\{\tau_k\leqslant T\}\geqslant\epsilon\;
\; \mbox{for all}\;\; k\geqslant k_1.
\end{equation}
\par
Let $p>0$, and
Define a $C^2$-function $V: \R^+\times\M\ra\R^+$ by
\[ V(x,i)=x^p-1-p\log x,\]
 then  $V(x,i)\geqslant0$ for all $x>0$ and $i\in M$. In fact, $u-1-\ln u\geqslant0$ for $u>0$.
\[\fr{\partial V}{\partial x}(x,i)=px^{p-1}-px^{-1};\; \fr{\partial^2 V}{\partial x^2}(x,i)=p(p-1)x^{p-2}+px^{-2}.\]
By It$\hat{o}$ formula, we have
\[\aligned
\d V(x,i)%&=V_xdx+1/2V_{xx}(dx)^2\\
&=(px^{p-1}-px^{-1})[(a(i)-b(i)x)dt+\sigma x^{\theta(i)}dB_t]+0.5 p\sigma^2(i)[(p-1)x^{p-2}+x^{-2}] x^{2\theta(i)}dt\\
&=pF(x)dt+\sigma(i) p(x^{p+\theta(i)-1}-x^{\theta(i)-1})dB_t,
\endaligned
\]
where $F(x)=a(i)(x^{p-1}-x^{-1})+b(i)(1-x^{p})+0.5\sigma^2(i)[(p-1)x^{p+2\theta(i)-2}+x^{2\theta(i)-2}]$, which is continuous on $x\in \R^+.$
\par Then we claim that the function $F$ is bounded above on $x\in \R^+$. In fact, for the case of $\theta(i)=0.5$, we take $p>1$ and obtain
 \[\aligned
 F(x)&=a(i)(x^{p-1}-x^{-1})+b(i)(1-x^{p})+0.5\sigma^2(i)[(p-1)x^{p-1}+x^{-1}]\\
 &=-[a(i)-0.5\sigma^2(i)]x^{-1}+[0.5\sigma^2(i)(p-1)+a(i)]x^{p-1}+b(i)(1-x^{p}),
 \endaligned
 \]
 and the higher power of $x$ is $p$ and the lower power of $x$ is $-1$, which together with $a(i)\geqslant0.5\sigma^2(i)$ yields
  \[\aligned
  &F(x)\thicksim -b(i)x^{p}\ra -\infty \;\; \mbox{as} \;x\ra +\infty,\;\; \mbox{and} \;\;F(x)\thicksim -a(i)x^{-1}\ra -\infty \;\; \mbox{as}\; x\ra 0^+\;\; \mbox{for}\;\; a(i)>0.5\sigma^2(i)\\
  &\mbox{or}\;\;F(x)\ra  b(i)\;\; \mbox{as}\; x\ra 0^+\;\; \mbox{while}\;\; a(i)=0.5\sigma^2(i).
  \endaligned
  \]
  \par For the case of $\theta(i)\in(0.5,1)$. We note that the higher power and lower power of $x$ in $F(x)$ are $p$  and  $-1$, respectively; and the coefficients of these two terms are negative, which yield
 \[F(x)\thicksim -b(i)x^{p}\ra -\infty \;\; \mbox{as} \;x\ra +\infty\;\; \mbox{and} \;\;F(x)\thicksim -a(i)x^{-1}\ra -\infty \;\; \mbox{as}\; x\ra 0^+.\]
 \par For the case of $\theta(i)\in[1,+\infty)$, that is $2\theta(i)-2\geqslant0.$ We see that the higher power and lower power of $x$ in $F(x)$ are $p+\theta(i)-2$  and  $-1$, respectively. Let $p\in(0,1)$, we get
 \[F(x)\thicksim -0.5\sigma(i)^2(1-p)x^{p+2\theta(i)-2}\ra -\infty \;\; \mbox{as} \;x\ra +\infty\;\; \mbox{and} \;\;F(x)\thicksim -a(i)x^{-1}\ra -\infty \;\; \mbox{as}\; x\ra 0^+.\]
\par Thus the continuation of $F$ in $\R^+$ implies that there must exist a positive constant $K$  such that $F(x)\leqslant K$ for all $x\in \R^+.$ Therefore we obtain
\[\aligned
\d V(x,i)\leqslant pKdt+\sigma(i) p(x^{p+\theta(i)-1}-x^{\theta(i)-1})dB_t.
\endaligned
\]
Integrating both sides of the above inequality from $0$ to $T\wedge\tau_k$, and taking expectations give
\begin{equation}\label{n2}
EV(X_{T\wedge\tau_k},r_{T\wedge\tau_k})\leqslant V(x_0,i_0)+pKE(T\wedge\tau_k).
\end{equation}
Set $\Omega_k=\{\tau_k\leqslant T\}$ for $k\geqslant k_1$, then it follows from \eqref{n1} that $P(\Omega_k)\geqslant\epsilon$. Note that for every $\omega\in \Omega_k,$ there is $X_{\tau_k}(\omega)$ equals either $k$ of $1/k$, and hence
 \[V(X_{\tau_k}(\omega),r_{\tau_k})\geqslant (k^p-1-p\log k)\wedge(k^{-p}-1+p\log k),\]
 which together with \eqref{n2} yields
 \[\infty>V(x_0,i_0)+pKT\geqslant E[I_{\Omega_k}V(X_{\tau_k}(\omega),r_{\tau_k})]\geqslant (k^p-1-p\log k)\wedge(k^{-p}-1+p\log k).\]
 Then by letting $k\ra\infty$, we get a contradiction. Thus $\tau_\infty=\infty$ a.s.
\end{proof}

\section{$p$th Moment Estimation}
In this section, we will give the $p$th moment estimation of  the solution to equation (\ref{w1}), and then present some sufficient conditions for asymptotic boundedness in $p$th moment of the solution $X_t(x_0, i_0)$.
\begin{Def}(Mao and Yuan (2006)\cite{mao-06})\par
A square matrix $A$ is called a nonsingular M-matrix if $A$ can be expressed in the form $A=sI-G$ with some $G\geqslant0$ (that is  each element of $G$ is non-nagative) and $s>\rho(G)$, where $I$ is the identity matrix and $\rho(G)$ the spectral radius of $G$.
\end{Def}

Corresponding to the infinitesimal generator $Q$  and $p \gt1$, we define an $m¡Á\times m$ matrix
\[\A(p):=p\,\text{diag}(b(1),\cdots,b(m))-Q.\]
\begin{lem}(Mao and Yuan (2006)\cite{mao-06})  \label{lem}\par
If $\A(p)$ is a nonsingular $M$-matrix, then there is a vector $(\beta_1,\cdots,\beta_m)^T$ with $\beta_i>0$ such that $
\mu_i:=pb(i)\beta_i-\sum_{k\in\M}q_{ik}\beta_k>0$ for all $i\in\M.
$
\end{lem}

\begin{theo}\label{th3.1}
 If
 one of the following conditions holds:
\par  1) $\A(1)$ is a nonsingular $M$-matrix;
 \par  2) $\forall i\in\M$, $\theta(i)\in[0,1]$, and $\A(p)$ is a nonsingular $M$-matrix with $p>\max\{1,2[1-\min_{i\in\M}\theta(i)]\}$.
 \\ Then the $p$th moment of the solution to equation \eqref{w1} has the following property,
\[
\E[X_{t}^p(x_0,i_0)]\lt  x_0^p e^{-\lambda_p \,t} +\fr {C_p}{\hat\beta\lambda_p}(1-e^{-\lambda_p t})\;\;\mbox{ for any } (x_0, i_0)\in\R^+\times\M,
\]
 where $C_p=\max_{i\in\M}\left\{\fr1p(p\beta_ia(i))^p\biggl(\frac{3}{\mu_i}\biggr)^{p-1}
+\fr{2-2\theta(i)}p\biggl(\frac12p(p-1)\sigma^2(i)\beta_i\biggr)^{p/[2-2\theta(i)]}\biggl(\frac{3}{\mu_i}\biggr)^{[p-2+2\theta(i)]/[2-2\theta(i)]}\right\}$, \par\quad$\;\;\,\hat\beta=\min_{i\in\M}{\beta_i}$ and $\lambda_p=\left\{
    \begin{array}{ll}
      \min_{i\in\M}\fr{[p+3-2\theta(i)]\mu_i}{3p\beta_i}, & p>1; \\
     \min_{i\in\M}\fr{\mu_i}{\beta_i}, & p=1.
    \end{array}
  \right.$
\par
Moreover, the solution of equation \eqref{w1} is asymptotic boundedness in  $p$th moment with
\[
\limsup_{t\ra\infty}\E[X_t^p(x_0, i_0)]\lt \fr{C_p}{\hat \beta\lambda_p}
\]and
the Lyapunov exponent
\[\limsup_{t\ra \infty}\fr1t\log \E[X_t^p(x_0,i_0)]\leqslant0,\;\;\mbox{for all}\; (x_0, i_0)\in \R^+\times\M.\]
\end{theo}

\begin{proof}
Lemma \ref{lem} yields that there is a vector $(\beta_1,\cdots,\beta_m)^T$ with $\beta_i>0$ such that
\[
\mu_i:=pb(i)\beta_i-\sum_{k\in\M}q_{ik}\beta_k>0$ for all $i\in\M.
\]
Define the $C^2$-function $V: \R^+\times\M\ra\R^+$ by the form:
\[V(x, i)=\beta_ix^p,\]
then $\fr{\partial V}{\partial x}(x,i)=p\beta_ix^{p-1}$ and $\fr{\partial^2 V}{\partial x^2}(x,i)=p(p-1)\beta_ix^{p-2}$, and we get
\begin{equation}\label{c22}
\aligned
\L V(x, i)=&\,p\,\beta_i[a(i)-b(i)x]x^{p-1}+\frac12p(p-1)\sigma^2(i)\beta_ix^{p+2\theta(i)-2}+\sum_{k\in\M}q_{ik}\beta_kx^p\\
=&\,-\bl(p\,b(i)\beta_i-\sum_{k\in\M}q_{ik}\beta_k\br)x^p+p\beta_ia(i)x^{p-1}+\frac12p(p-1)\sigma^2(i)\beta_ix^{p+2\theta(i)-2}\\
\lt &-\mu_ix^p+ p\,\beta_ia(i)x^{p-1}+\frac12p(p-1)\sigma^2(i)\beta_ix^{p+2\theta(i)-2}\;\;\mbox{ for all}\;\; (x, i)\in \R^+\times\M.
\endaligned
\end{equation}
By the elementary inequality
\[
a^\gamma b^{1-\gamma}\lt a \gamma+b (1-\gamma), \;\;\forall \, a, b\geqslant0, \gamma\in[0,1],
\]
 we obtain
 \begin{equation}\label{c2}
\aligned
 p\,\beta_ia(i)x^{p-1}= &p\,\beta_ia(i)\biggl(\frac{3}{\mu_i}\biggr)^{(p-1)/p}\biggl(\frac{\mu_i}{3}x^p\biggr)^{(p-1)/p}\\
 =&\left[(p\,\beta_ia(i))^p\biggl(\frac{3}{\mu_i}\biggr)^{(p-1)}\right]^{1/p}\biggl(\frac{\mu_i}{3}x^p\biggr)^{(p-1)/p}\\
 \lt& \fr1p(p\,\beta_ia(i))^p\biggl(\frac{3}{\mu_i}\biggr)^{(p-1)}+\fr{p-1}{3p}\mu_i x^p
 \endaligned
\end{equation}
and
 \begin{equation}\label{c3}
\aligned
 &\frac12p(p-1)\sigma^2(i)\beta_ix^{p+2\theta(i)-2}\\
 =&\frac12p(p-1)\sigma^2(i)\beta_i\biggl(\frac{3}{\mu_i}\biggr)^{[p-2+2\theta(i)]/p}\biggl(\frac{\mu_i}{3}x^p\biggr)^{[p-2+2\theta(i)]/p}
 \\
 =&\left[\biggl(\frac12p(p-1)\sigma^2(i)\beta_i\biggr)^{p/[2-2\theta(i)]}\biggl(\frac{3}{\mu_i}\biggr)^{[p-2+2\theta(i)]/[2-2\theta(i)]}\right]^{[2-2\theta(i)]/p}\biggl(\frac{\mu_i}{3}x^p\biggr)^{[p-2+2\theta(i)]/p}\\
 \lt& \fr{2-2\theta(i)}p\biggl(\frac12p(p-1)\sigma^2(i)\beta_i\biggr)^{p/[2-2\theta(i)]}\biggl(\frac{3}{\mu_i}\biggr)^{[p-2+2\theta(i)]/[2-2\theta(i)]}+\fr{p-2+2\theta(i)}{3p}\mu_i x^p.
 \endaligned
\end{equation}
Substituting  \eqref{c2} and \eqref{c3} into \eqref{c22} gives
\begin{equation}\label{c23}
\aligned
\L V(x, i)%=&\,p\beta_i[a(i)-b(i)x]x^{p-1}+\frac12p(p-1)\sigma^2(i)\beta_ix^{p+2\theta(i)-2}+\sum_{k\in\M}q_{ik}\beta_kx^p\\
%=&\,-\bl(pb(i)\beta_i-\sum_{k\in\M}q_{ik}\beta_k\br)x^p+p\beta_ia(i)x^{p-1}+\frac12p(p-1)\sigma^2(i)\beta_ix^{p+2\theta(i)-2}\\
\lt &-\mu_ix^p+\fr{p-1}{3p}\mu_i x^p+\fr{p-2+2\theta(i)}{3p}\mu_i x^p+ \fr1p(p\beta_ia(i))^p\biggl(\frac{3}{\mu_i}\biggr)^{(p-1)}
\\
\quad&+\fr{2-2\theta(i)}p\biggl(\frac12p(p-1)\sigma^2(i)\beta_i\biggr)^{p/[2-2\theta(i)]}\biggl(\frac{3}{\mu_i}\biggr)^{[p-2+2\theta(i)]/[2-2\theta(i)]}\\
\lt  &-\lambda_p V(x,i)+C_p\;\; \mbox{for all}\;  (x, i)\in \R^+\times\M,
\endaligned
\end{equation}

Define the stopping time $\tau_n=\inf\{t\gt0|\,X_t\gt n\}$, it is obvious that $\tau_n\ra\infty$ as $n\ra\infty$.  By applying It$\hat{\text{o}}$'s
formula, we have
\[
\E[e^{\lambda_p \,t_n}V(X_{t_n},r_{t_n})]=V(x_0, i_0)+\E\int_0^{t_n}e^{\lambda_p s}\L V(X_s,  r_s)ds+\lambda_p\E\int_0^{t_n}e^{\lambda_p s}V(X_s, r_s)ds,
\]
where $t_n=\tau_n\wedge t.$ It follows from  \eqref{c23}  that
\[
\E[e^{\lambda_p\,t_n}V(X_{t_n}, r_t)]\lt V(x_0, i_0)+C_p\int_0^{t_n}e^{\lambda_p s}ds\lt  V(x_0, i_0)+\fr{C_p}{\lambda_p}(e^{\lambda_p t_n}-1),
\]
and thus
\[
\hat \beta  \E[X_{t_n}^p(x_0,i_0)]\lt  \beta_{i_0}x_0^pe^{ -\lambda_p\,t_n}+\fr{C_p}{\lambda_p}(1-e^{-\lambda_p t_n}).
\]
Letting $n\ra\infty$, we get
\[
\E[X_{t}^p(x_0,i_0)]\lt  x_0^p e^{-\lambda_p \,t} +\fr {C_p}{\hat\beta\lambda_p}(1-e^{-\lambda_p t}),\;\;\mbox{for any}\;\; (x_0, i_0)\in\R^+\times\M.
\]
 Thus
$\limsup_{t\ra\infty}\E[X_t^p(x_0,i_0)]\lt \fr{C_p}{\hat\beta\lambda_p},$ which gives
 \[\limsup_{t\ra \infty}\fr1t\log \E[X_t^p(x_0,i_0)]\leqslant0\;\; \mbox{for all}\;\; (x_0, i_0)\in \R^+\times\M.\]
\end{proof}

\vspace{2mm}
\section{Stationary distribution}
In this section, we assume that $q_{ij}>0$ for $i\not=j$, and the discrete component $(r_t)_{t\gt0} $ in equation \eqref{w1} is an irreducible
continuous-time Markov chain with an invariant distribution $\pi = (\pi_1, \pi_2,...,\pi_m)$.

\begin{lem}(See Zhu and Yin (2007)\cite{zhu-07}) \label{mm}
\par
  If there is a bounded open
subset $D\subset\R^+$ and for each $i\in \M$ there exists a nonnegative function
$V(\cdot,i): D^c \rightarrow \R $ such that $V(\cdot,i)$ is twice continuously differentiable
and for some $\xi >0$,
$\L V(x,i)\leqslant -\xi\; \mbox{ for any } \; (x, i)\in D^c \times\M,$
then equation (\ref{w1}) is positive recurrence.
 Moreover, the  process $X_t(x,i)$ has a unique ergodic stationary
distribution $\nu$. That is, if $f$ is a function integrable with respect to the
measure $\nu$, then
\[P\bigg(\lim_{t\ra\infty}\fr1t\int_0^tf(x(s))ds=\int_0^\infty f(x)\nu(dx)\bigg)= 1.\]
\end{lem}

\begin{theo}\label{sdt}
If   $2\theta(i)\in (1, 3]$ for all $i\in\M$,  then for any $(x_0,i_0)\in \R^+\times\M,$  equation (\ref{w1}) is positive recurrence and the process $X_t(x_0,i_0)$ admits a unique ergodic  stationary distribution $\nu.$
\end{theo}
\begin{proof}
Define a $C^2$-function $V: \R^+\times \M\ra \R^+$ by the  form:
\[
V(x,i)=(\gamma-p\xi_i)x^{-p}+x,
\]
where $p, \gamma$ are positive constants satisfying $\min_{i\in\M}\{\gamma-p\,\xi_i\}>0$ where $\xi=(\xi_1,\cdots,\xi_m)^T$ is a solution of the following Poisson system,
\begin{equation}\label{poissoneq}
Q\xi=-\mu+\sum_{i\in\M}\pi_i\mu(i)\textbf{1},
\end{equation}
where $\mu=(\mu(1),\mu(2),...,\mu(m))$. By applying Ito's formula, we obtain
\[
\aligned
\L V(x,i)=&-pa(i)(\gamma-p\xi_i)x^{-p-1}+pb(i)(\gamma-p\xi_i)x^{-p}+a(i)-b(i)x\\
&+0.5p(1+p)(\gamma-p\xi_i)\sigma^2(i)x^{2\theta(i)-p-2}-px^{-p}\sum_{k\in\M} q_{ik}\xi_k.
\endaligned
\]
Then, it follows from $1<2\theta(i)\leq3$  that
\[
\L V(x,i)\thicksim_{x\ra+\infty} -b(i)x\;\; \mbox{and}
\;\; \L V(x,i)\thicksim_{x\ra\, 0^+} -p\,a(i)(\gamma-p\xi_i)x^{-p-1}.
\]
Set $D=(1/N,N)\subset  \R^+,$ then for sufficiently large $N$ we get
\[
\L V(x,i)\leq -1, \;\;\mbox{for all}\;\; (x,i)\in D^c\times\M.
\]
Then,  Lemma \ref{mm} shows that equation (\ref{w1}) admits a  stationary distribution with nowhere-zero density in $\R^+.$
\end{proof}

\begin{theo}\label{th4.1} If $\theta(i)\equiv1/2$ and $\mu(i):=a(i)-0.5\sigma^2(i)>0$ for all $i\in\M$, then, for any $(x_0,i_0)\in \R^+\times\M,$ the solution $X_t(x_0,i_0)$ of equation (\ref{w1}) is positive recurrence and admits a unique ergodic  stationary  distribution $\nu.$
\end{theo}

\begin{proof}
 Define the $C^2$-function $V: \R^+\times \M\ra \R^+$ by the following form
\[V(x,i)=(\gamma-p\xi_i)x^{-p}+x,\]
where $p, \gamma$ are positive constants satisfying $\min_{i\in\M}\{\gamma-p\xi_i\}>0$, and $\xi=(\xi_1,\cdots,\xi_m)^T$ is a solution of the Poisson system \eqref{poissoneq}. By using It$\hat{\text{o}}$'s formula, we obtain
\[
\aligned
\L V(x,i)=&-p(\gamma-p\xi_i)x^{-p-1}\left[\mu(i)-0.5p\sigma^2(i)\right]+pb(i)(\gamma-p\xi_i)x^{-p}
+a(i)x-b(i)x-px^{-p}\sum_{k\in\M} q_{ik}\xi_k.
\endaligned
\]
 Meanwhile, $\mu(i)>0$ yields one can choose sufficient small $p$ such that $\mu(i)-0.5p\sigma^2(i)>0$ for all $i\in\M$,
 then
 \[
\L V(x,i)\thicksim_{x\ra+\infty} -b(i)x\;\; \mbox{and}
\;\; \L V(x,i)\thicksim_{x\ra\, 0^+} -p(\gamma-p\xi_i)\left[\mu(i)-0.5p\,\sigma^2(i)\right]x^{-p-1},
\]
which together with Lemma \ref{mm} give the result immediately.
\end{proof}

\vspace{2mm}
\section{Example}
In this section, we will give an example to verify the theorems obtained in previous sections by simulation. The numerical method used here is Milstein¡¯s Higher Order Method, see Higham (2001) \cite{Higham-01} for more details.
 \par Set the state space of Markov chain $(r_t)_{t\geqslant0}$ as $\M=\{1,\,2,\,3,\,4\}$, and its  generator $Q$ as follows
\[Q=\left(
      \begin{array}{cccc}
      -7 & 3 & 2 & 2 \\
        3 & -9 & 3 & 3\\
        3 & 2 & -8 & 3\\
        2 & 3 & 4 & -9\\
      \end{array}
    \right).
   \]
   Then its stationary distribution is
$\pi=( 0.2773,\,   0.2277,\,    0.2681,\,    0.2269).$ Its one step transition probability matrix $P(\delta)$ with $\delta=10^{-4}$ is
   \[
P(\delta)=\exp(\delta  Q)=\left(
      \begin{array}{cccc}
     0.9993  &  0.0003 &   0.0002 &   0.0002\\
    0.0003  &  0.9991 &   0.0003  &  0.0003\\
    0.0003  &  0.0002 &   0.9992  &  0.0003\\
    0.0002  &  0.0003 &   0.0004   & 0.9991\\
      \end{array}
    \right).
 \]

\par Let $a=(4,1.5,0.8,0.55)$, $b=(2,1,1,1)$, $\sigma=(2,1,1,0.4)$, $\theta=(0.5, 0.5, 0.5, 0.5)$ and $p=2$, then the conditions $2a(i)\geqslant\sigma^2(i),(i=1,2,3,4)$ hold.
\[
\A(p)=p\, \mbox{diag}(b_1,b_2,b_3,b_4)-Q=\left(
                                           \begin{array}{cccc}
                                             11 & -3 & -2 & -2 \\
                                             -3 & 11 & -3 & -3 \\
                                             -3 & -2 & 11 & -3 \\
                                             -2 & -3 & -4 & 11 \\
                                           \end{array}
                                         \right)=sI-G,
\mbox{ where } s=11,\;
G=\left(
                                     \begin{array}{cccc}
                                             0 & 3 & 2 & 2 \\
                                             3 & 0 & 3 & 3 \\
                                             3 & 2 & 1 & 3 \\
                                             2 & 3 & 4 & 0 \\
                                           \end{array}
                                         \right)
\] and $\rho(G)=8.5208$.
 Thus from Theorem \ref{Th1} we see that $X_t(x_0, i_0)>0$ for $t\geqslant0$ a.s.. Meanwhile, $\A(2)$ is a nonsingular $\M$-matrix, then it follows from Theorem \ref{th3.1} that $X_t(x_0, i_0)$ of equation (\ref{w1}) is asymptotic boundedness in mean-square, and the Lyapunov exponent of $X_t(x_0, i_0)$ is no more than 0. Moreover, according to Theorem \ref{th4.1} the solution $X_t(x_0, i_0)$ of equation (\ref{w1}) is positive recurrence and admits a unique ergodic  stationary  distribution $\nu.$
These claims are supported by Figures 1, 2 and 3, respectively.
\begin{center}
{\includegraphics[scale=0.7]{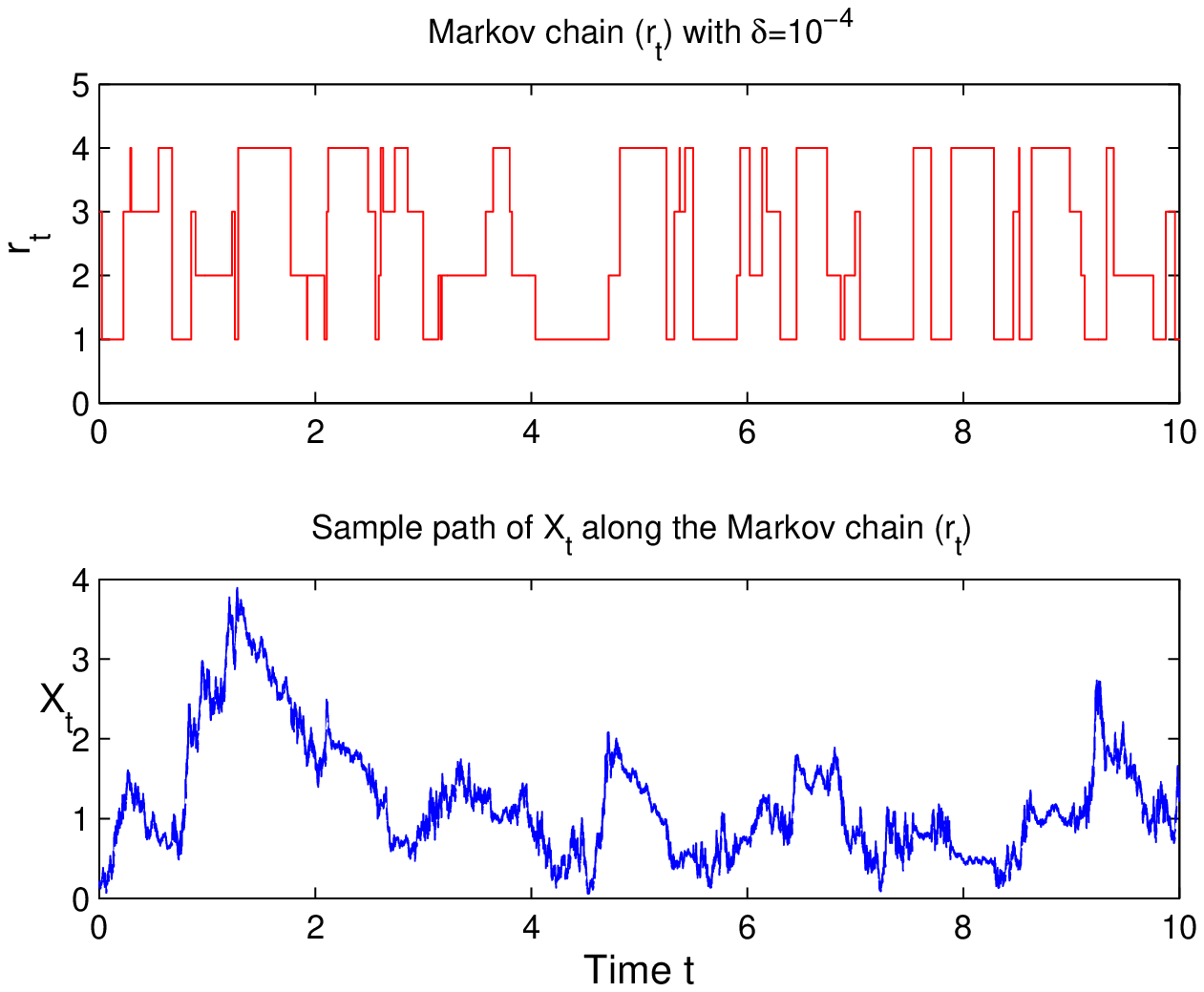}}\\
\footnotesize {\textbf{Figure 1.}\;\;
Sample path of $X_t$ (in blue line) along the Markov chain $(r_t)_{t\geqslant0}$ (in red line) with initial value $(x_0,r)=(0.2,3)$.}
\end{center}

\begin{center}
{\includegraphics[scale=0.7]{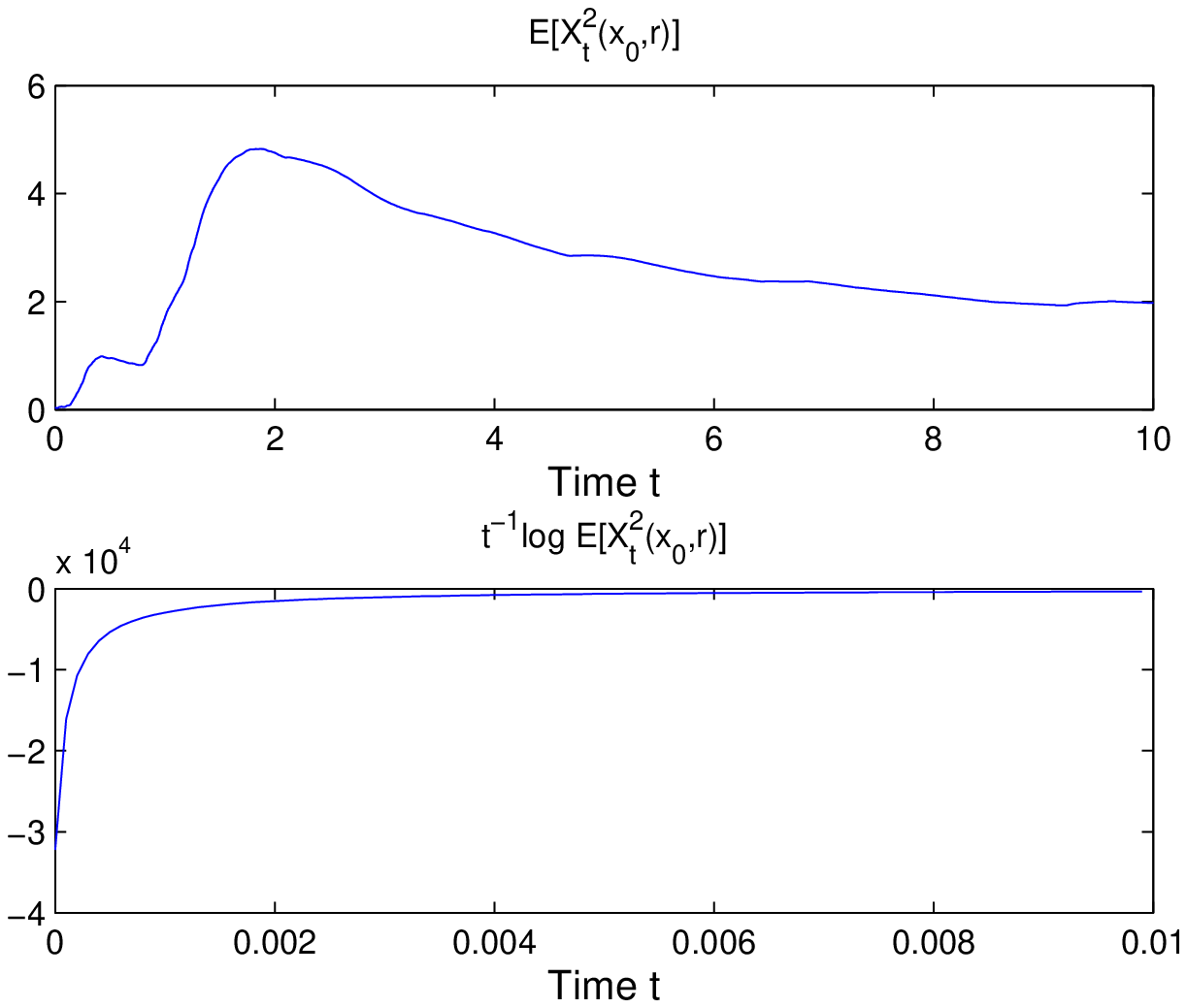}}\\
\footnotesize {\textbf{Figure 2.}\;\;
Mean-square value of $X_t$ and  $\frac1t \log E\left[X_t^2(x_0,r)\right]$ with initial value $(x_0,r)=(0.2,3)$.
}
\end{center}

\begin{center}
{\includegraphics[scale=0.7]{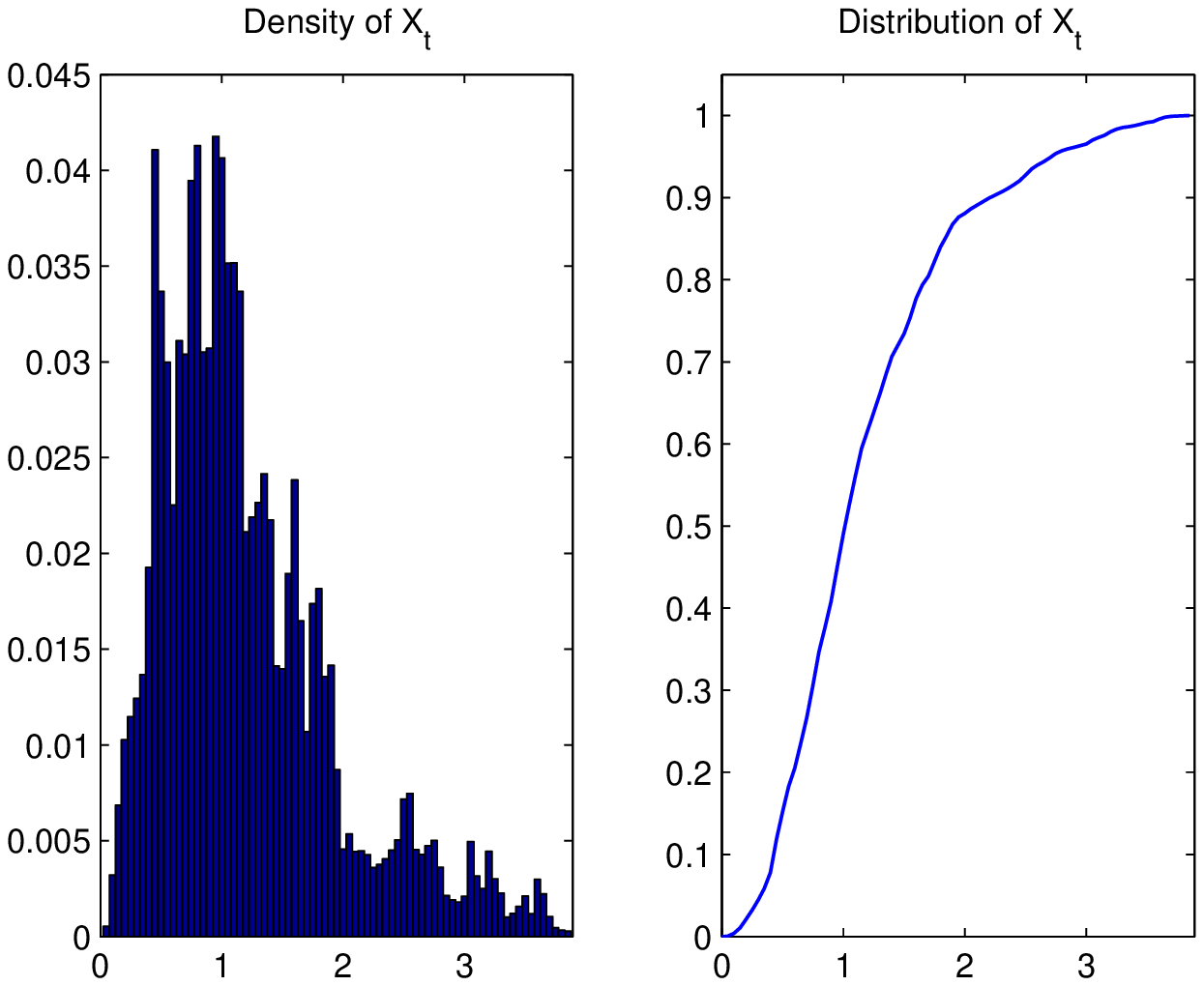}}\\
\footnotesize {\textbf{Figure 3.}\;\;
Density and distribution of $X_t$  with initial value $(x_0,r)=(0.2,3)$.}
\end{center}
\vspace{2mm}
\section{Conclusion}
In this paper, we investigate  a general stochastic volatility equation with regime switching, and show that for any initial value $(x_0,i_0)\in \R^+\times\M $ there is a unique global almost surely positive solution $X_t(x_0, i_0)$ to this equation; and give the $p$th moment estimation of  the solution by using the properties of nonsingular $M$-matrix, and present some simple  sufficient conditions for the existence of a unique ergodic stationary distribution of the equation.
 \par It will be more interesting if the Markov chain $(r_t)_{t\geqslant0}$ is state dependent or infinity, also the equation with jumps will be better than the one without jumps in describing some complicated dynamics behaviors in the real world.

%--------------------------the references---------------------------------------------------
\vspace{5mm}

\bibliographystyle{apalike}
\end{document}